\documentclass[12pt]{amsart}
\numberwithin{equation}{section}
\usepackage{amssymb}
\usepackage{amsmath}
\allowdisplaybreaks[3]
\newcommand{\eps}{\varepsilon}
\def\R{\mathbb R}
\def\C{\mathbb C}
\def\N{\mathbb N}
\def\arg{\operatorname{arg}}
\def\meas{\operatorname{meas}}
\newtheorem{lemma}{Lemma}[section]
\newtheorem{theorem}{Theorem}[section]
\theoremstyle{definition}
\newtheorem{definition}{Definition}[section]
\theoremstyle{remark}
\newtheorem*{rem}{Remark}
\newtheorem{remark}{Remark}
\begin{document}
\title[Wiman-Valiron disks]{The size of Wiman-Valiron disks}
\dedicatory{Dedicated to Professor C.-C.\ Yang on the occasion
of this 65th birthday}
\author{Walter Bergweiler}
\thanks{Supported by the G.I.F.,
the German--Israeli Foundation for Scientific Research and
Development, Grant G-809-234.6/2003,
the EU Research Training Network CODY,
and the ESF Research Networking Programme HCAA}
\address{Mathematisches Seminar,
Christian--Albrechts--Universit\"at zu Kiel,
Lude\-wig--Meyn--Str.~4,
D--24098 Kiel,
Germany}
\email{bergweiler@math.uni-kiel.de}
\begin{abstract}
Wiman-Valiron theory and results of Macintyre about ``flat regions''
describe the asymptotic behavior of entire functions
in certain disks around points of maximum modulus.
We estimate the size of these disks for Macintyre's theory
from above and below.
\end{abstract}
\maketitle
\section{Introduction}
Let
$f(z)=\sum_{n=0}^\infty a_nz^n$
be an entire function,
$M(r,f):=\max_{|z|=r}|f(z)|$ its {\em maximum modulus}
and
$\mu(r,f):=\max_{n\geq 0} |a_n|r^n$  its {\em maximum term}.
The largest $n$ for which $\mu(r,f)=|a_n|r^n$
is denoted by $\nu(r,f)$  and called the {\em central index}.
(Except for a discrete set of $r$-values there is only one
integer $n$ with $\mu(r,f)=|a_n|r^n$.)
We say that a set $F\subset [1,\infty)$
has {\em finite logarithmic measure} if $\int_F dt/t <\infty$.

The main result of Wiman-Valiron
theory says that there exists a set
$F$ of finite logarithmic measure
such that if $|z_r|=r\notin F$,  if $|f(z_r)|=M(r,f)$ and
if $z$ is sufficiently close to~$z_r$, then
\begin{equation}\label{wv}
f(z)\sim \left(\frac{z}{z_r}\right)^{\nu(r,f)}f(z_r)
\end{equation}
as $r\to\infty$.
Equivalently,
$$f(e^\tau z_r)\sim e^{\nu(r,f)\tau}f(z_r)$$
if $|\tau|$ is sufficiently small.
Wiman~\cite{Wiman1916} obtained~\eqref{wv} for
$$|z|=r \quad\text{and}\quad
\left|\arg z -\arg z_r\right|\leq \frac{1}{\nu(r,f)^{3/4+\delta}}$$
if $\delta>0$ while
Valiron~\cite[Theorem~29]{Valiron1923} proved~\eqref{wv} under
the conditions
$$\bigl||z|-r\bigr|\leq \frac{Kr}{\nu(r,f)}\quad\text{and}\quad
\left|\arg z -\arg z_r\right|\leq \frac{1}{\nu(r,f)^{15/16}},$$
for any
given constant~$K$.
Macintyre~\cite{Macintyre1938} noted that~\eqref{wv} holds for
$$|z-z_r|\leq \frac{r}{\nu(r,f)^{1/2+\eps}}$$
if $\eps>0$.
The sharpest estimates are due to Hayman~\cite{Hayman1974} whose results
imply that if
$$\psi(t)=t \cdot \log t \cdot \log \log t \cdot \ldots  \cdot\log^{m-1}t
\cdot \left(\log^m t\right)^{1+\eps},$$
where $\eps>0$, $m\in\N$ and
$\log^m$ denotes the $m$-th iterate of the logarithm,
then~\eqref{wv} holds
for
$$|z-z_r|\leq \frac{r}{\sqrt{\psi(\nu(r,f)) \log \psi(\nu(r,f))}}.$$

Results similar to those of Wiman-Valiron theory
were obtained by Macintyre~\cite{Macintyre1938} with $\nu(r,f)$
replaced by
$$a(r,f):=\frac{d \log M(r,f)}{d\log r}.$$
Recall here that $\log M(r,f)$ is convex in $\log r$. Since
convex functions have nondecreasing left and right derivatives
and since they are differentiable except for an at most countable set,
the derivative of $\log M(r,f)$ with respect to $\log r$
exists except possibly for a countable set of $r$-values.
(Actually,
by a result of Blumenthal (see~\cite[Section~II.3]{Valiron1923}),
the set of $r$-values where $\log M(r,f)$ is not differentiable
is discrete.) To be definite, we shall always denote
by $a(r,f)$ the right derivative of $\log M(r,f)$ with
respect to $\log r$. Then $a(r,f)$ is nondecreasing and
it can be shown that
$$
a(r,f)= \frac{z_r f'(z_r)}{f(z_r)}
$$
except for an at most countable set of $r$-values.
The result of
Macintyre~\cite[Theorem~3]{Macintyre1938} says that
$$
f(z)\sim \left(\frac{z}{z_r}\right)^{a(r,f)}f(z_r)
$$
for
\begin{equation}\label{Meps}
|z-z_r|\leq \frac{r}{(\log M(r,f))^{1/2+\eps}}
\end{equation}
as $r\to\infty$, $r\notin F$.

More recently, a result of this type was obtained in~\cite{BRS}.
There it is not required that $f$ is entire but only that $f$ is
as in the following definition.
\begin{definition} \label{defin1}
Let $D$ be an unbounded domain in $\C$ whose boundary consists
of piecewise smooth curves. Suppose that
the complement of $D$ is unbounded.
Let $f$ be a complex-valued function whose domain of definition
contains the closure $\overline{D}$ of~$D$. Then $D$ is called a
{\em direct tract} of $f$ if $f$ is holomorphic in $D$ and
continuous in $\overline{D}$ and if there exists $R>0$ such that
$|f(z)|=R$ for $z\in\partial D$ while
$|f(z)|>R$ for $z\in D$.
\end{definition}
We note that
every transcendental entire function has a direct tract.
Let $f,D,R$ be as in the above definition and put
$$
\label{defM}
M(r,f,D):=\max_{|z|=r, z\in D}|f(z)|.
$$
Then $\log M(r,f,D)$ is again convex in $\log r$.
Denoting by $a(r,f,D)$ the right derivative of $\log M(r,f,D)$ with
respect to $\log r$ we see as before
that $a(r,f,D)$ is nondecreasing and
$$
a(r,f,D)
= \frac{z_r f'(z_r)}{f(z_r)}
$$
except for an at most countable set
of $r$-values, with $z_r\in D$ such that
$|z_r|=r$ and $|f(z_r)|=M(r,f,D)$.
It follows from a  result of Fuchs~\cite{Fuchs1981} that
$$
\lim_{r\to\infty} \frac{\log M(r,f,D)}{\log r} =\infty
\quad \text{and} \quad
\lim_{r\to\infty} a(r,f,D)=\infty.
$$
The main result of~\cite{BRS} says that if $\tau>\frac12$, then
there exists a set $F$ of finite logarithmic
measure such that
\begin{equation}\label{wvlike}
f(z)\sim \left(\frac{z}{z_r}\right)^{a(r,f,D)}f(z_r)
\end{equation}
for
\begin{equation}\label{wvsize}
|z-z_r|<\frac{r}{a(r,f,D)^{\tau}}
\end{equation}
as $r\to\infty$, $r\notin F$.
In particular, the disk of radius $r/a(r,f,D)^{\tau}$ around
$z_r$ is contained in the direct tract~$D$.

We investigate the question how large the disk around $z_r$ in
which~\eqref{wvlike}
holds can be chosen. Our main result says that
if $\psi:[t_0,\infty)\to (0,\infty)$ satisfies certain regularity
conditions discussed below, then (\ref{wvlike}) holds
for $|z-z_r|<r/\sqrt{\psi(a(r,f,D))}$
if
\begin{equation}\label{conv}
\int_{t_0}^\infty \frac{dt}{\psi(t)} <\infty
\end{equation}
and if $r\notin F$ is sufficiently large, but~(\ref{wvlike})
need not hold in this disk if
\begin{equation}\label{div}
\int_{t_0}^\infty \frac{dt}{\psi(t)} =\infty.
\end{equation}
The ``interesting'' functions for conditions (\ref{conv}) and (\ref{div})
are functions like
$$\psi(t)=t\left(\log t\right)^\alpha$$
or, more generally,
$$\psi(t)=t \cdot \log t \cdot \log \log t \cdot \ldots  \cdot\log^{m-1}t
\cdot \left(\log^m t\right)^{\alpha},$$
where $\alpha>0$ and $m\in\N$.
Here (\ref{conv}) holds for $\alpha>1$ while (\ref{div}) holds for
$\alpha\leq 1$. For these functions we have
$$1\leq \frac{t\psi'(t)}{\psi(t)}\leq 1+o(1)$$
as $t\to\infty$. Therefore it does not seem to be
a severe restriction
to impose the condition that $\psi$ is differentiable and
satisfies
\begin{equation}\label{reg}
K\leq \frac{t\psi'(t)}{\psi(t)}\leq L
\end{equation}
for certain constants $K$ and $L$ satisfying $0\leq K\leq 1<L$.

Our results are as follows.
\begin{theorem} \label{th2}
Let $t_0>0$ and let
$\psi:[t_0,\infty)\to (0,\infty)$ be a differentiable function
satisfying~\textup{(\ref{conv})}
and~\textup{(\ref{reg})} for some $K>0$ and $L<2$.

Let $f$ be a function with a direct tract $D$ and
let $z_r\in D$ with $|z_r|=r$ and $|f(z_r)|= M(r,f,D)$.
Then there exists a set $F$ of finite logarithmic
measure such that
\begin{equation}\label{wvlikenew}
f(z)\sim \left(\frac{z}{z_r}\right)^{a(r,f,D)}f(z_r)
\quad \text{for} \quad
|z-z_r|\leq \frac{r}{\sqrt{\psi(a(r,f,D))}}
\end{equation}
as $r\to\infty$, $r\notin F$.
\end{theorem}
\begin{theorem} \label{th3}
Let $t_0>0$ and let
$\psi:[t_0,\infty)\to (0,\infty)$ be a differentiable function
satisfying~\textup{(\ref{div})}
and~\textup{(\ref{reg})} for $K=1$ and some $L<\frac65$.

Then there exists an entire function $f$ which has exactly one tract
$D$ such that if $r$ is sufficiently large and $|z|=r$, then
the disk of radius $r/\sqrt{\psi(a(r,f,D))}$ around $z$
contains a zero of~$f$.
\end{theorem}
In particular it follows under
the hypotheses of Theorem~\ref{th3} that
the disk mentioned
is not contained in $D$ and that (\ref{wvlikenew}) does
not hold.

\begin{remark}
Our method also yields that
if $f$ is entire and $z_r$ is a point of modulus~$r$
with $|f(z_r)|=M(r,f)$, then~\eqref{wvlikenew} holds with
$a(r,f,D)$ replaced by $a(r,f)$. Here we only note that
if $D_r$ is the direct tract containing $z_r$, then
$$a(r,f)=a(r,f,D_r)=\frac{z_r f'(z_r)}{f(z_r)}$$
except for an at most countable set of $r$-values.

We also note that if $\psi$ satisfies~\textup{(\ref{conv})},
then
\begin{equation}\label{alogM}
a(r,f,D)\leq \psi(\log M(r,f,D))
\end{equation}
outside a set of finite
logarithmic measure. In fact, if $s_0:=\log M(r_0,f,D)\geq t_0$
and  if $F$ denotes the set of all $r\geq r_0$ where~\eqref{alogM}
does not hold, then
$$
\int_F\frac{dt}{t} \leq \int_F\frac{a(t,f,D)}{\psi(\log M(t,f,D))} \frac{dt}{t}
\leq \int_{r_0}^\infty \frac{a(t,f,D)}{\psi(\log M(t,f,D))} \frac{dt}{t}
=\int_{s_0}^\infty \frac{dt}{\psi(t)}<\infty.$$
We deduce that
the condition $|z-z_r|\leq r/\sqrt{\psi(a(r,f,D))}$ in~\eqref{wvlikenew}
can be replaced by
$$|z-z_r|\leq \frac{r}{\sqrt{\psi(\psi(\log M(r,f,D)))}}.$$
For entire $f$ we can again
replace $M(r,f,D)$ by $M(r,f)$ if $|f(z_r)|=M(r,f)$.
With $\psi(t)=t^{1+\delta}$ we recover Macintyre's condition~\eqref{Meps}.
\end{remark}
\begin{remark} \label{rem1}
In the papers on Wiman-Valiron theory cited above it is
usually not required that $|f(z_r)|=M(r,f)$ but only
that $|f(z_r)|\geq \eta M(r,f)$ for some $\eta\in (0,1)$,
possibly depending on~$r$. It is then shown that~\eqref{wv} holds for
$z$ in some disk around~$z_r$ whose size depends on $\eta$.
In~\cite{BRS} only the case $\eta=1$ is considered, although the
method allows to deal with the case $0< \eta <1$ as well.
For the sake of simplicity we also restrict to the case $\eta=1$
in this paper.
\end{remark}
\begin{remark} \label{rem2}
It was shown in~\cite{Bergweiler1990a} that the estimate on the size
of the exceptional set~$F$ is best possible in Wiman-Valiron theory,
and it follows from the results there that this also holds for
Macintyre's theory and Theorem~\ref{th2}.
\end{remark}
\begin{remark} \label{rem3}
We do not discuss the numerous applications that the theories of
Wiman-Valiron and Macintyre have found, but just mention some
references with applications to complex  differential
equations~\cite{ELN,Jank1985,Laine1993,Wittich1968},
distribution of zeros of derivatives~\cite{Clunie1989,Langley1993},
and complex dynamics~\cite{BRS,Eremenko1989,Hua98}.
\end{remark}

\section{Proof of Theorem~\ref{th2}}
Let $D$ be a direct tract of~$f$.
The proof in~\cite{BRS} that~\eqref{wvlike}
holds for $z$ satisfying~\eqref{wvsize} relies on a lemma~\cite[Lemma~11.3]{BRS}
which says that if $\beta>\frac12$, then
there exists a set $F$ of finite logarithmic measure
such that
\begin{equation} \label{growthlemma3.1}
\log M(s,f,D) \leq \log M(r,f,D) +a(r,f,D)\log\frac{s}{r}+o(1)
\end{equation}
for
\begin{equation} \label{growthlemma3.2}
\left|\log\frac{s}{r}\right|\leq \frac{1}{a(r,f,D)^\beta},
\end{equation}
uniformly as $r\to\infty$, $r\notin F$.
In order to prove  Theorem~\ref{th2} we shall prove that
if $\psi$ satisfies the hypothesis of this theorem,
then~\eqref{growthlemma3.2} can be replaced by
\begin{equation} \label{growthlemma3a}
\left|\log\frac{s}{r}\right|\leq \frac{1}{\sqrt{\psi(a(r,f,D))}}.
\end{equation}
In order to prove that~\eqref{growthlemma3.1} holds under
the assumption~\eqref{growthlemma3a}
we use the following lemma.
\begin{lemma} \label{growthlemma1a}
Let $x_0>0$ and let $T:
[x_0,\infty)\to (0,\infty)$ be nondecreasing.
Let $t_0:=T(x_0)$ and let $\sigma_1,\sigma_2:[t_0,\infty)\to(0,\infty)$ be
nondecreasing functions such that
$$
\int_{t_0}^\infty\frac{dt}{\sigma_1(t)\sigma_2(t)}<\infty.$$
Suppose also that $\sigma_2$ is differentiable
and satisfies
$$
0\leq \frac{t\sigma_2'(t)}{\sigma_2(t)}\leq 1-\delta
$$
for $t\geq t_0$ and some $\delta >0$.
Then there exists a set $E\subset [x_0,\infty)$ of finite measure
such that if $x\notin E$, then
\begin{equation}
T\left(x+\frac{1}{\sigma_1(T(x))}\right) <
T(x) + \sigma_2(T(x))     \label{upperboundTa}
\end{equation}
and
\begin{equation}
T\left(x-\frac{1}{\sigma_1(T(x))}\right) >
T(x) - \sigma_2(T(x)) .     \label{lowerboundTa}
\end{equation}
\end{lemma}

\begin{proof}
First we note that
$x-1/\sigma_1(T(x))\geq x_0$ for
sufficiently large $x$, say $x\geq x_0'$.
Thus the left hand side
of~\eqref{lowerboundTa} is defined for  $x\geq x_0'$.
Denoting by $E_1$ the subset of $[x_0,\infty)$ where~\eqref{upperboundTa}
fails and by $E_2$ the subset of $[x_0',\infty)$
where~\eqref{lowerboundTa} fails we can thus
take $E=[x_0,x_0']\cup E_1\cup E_2$.

We put $G(t):=t/ \sigma_2(t)$.
Since
$$
\frac{tG'(t)}{G(t)}=1-\frac{t\sigma_2'(t)}{\sigma_2(t)}\geq \delta
$$
the function  $G$ is increasing and hence
\begin{eqnarray*}
G\left(t+\sigma_2(t)\right) - G(t) & = &\int_{t}^{t+\sigma_2(t)} G'(u)du\\
&\geq & \delta \int_{t}^{t+\sigma_2(t)} \frac{G(u)}{u}du\\
&\geq & \delta G(t) \int_{t}^{t+\sigma_2(t)} \frac{du}{u}\\
&= & \delta G(t) \log\left(1+\frac{1}{G(t)}\right)
\end{eqnarray*}
for $t\geq t_0$. Since the function $x\mapsto x \log\left(1+1/x\right)$
is increasing for $x>0$ we deduce that
\begin{equation} \label{G+}
G\left(t+\sigma_2(t)\right) - G(t)
\geq \eta:=  \delta G(t_0) \log\left(1+\frac{1}{G(t_0)}\right)>0
\end{equation}
for $t\geq t_0$.
Similarly,
\begin{eqnarray}
G(t) -G\left(t-\sigma_2(t)\right) & =& \int_{t-\sigma_2(t)}^{t} G'(u)du\nonumber\\
&\geq & \delta \int_{t-\sigma_2(t)}^{t} \frac{G(u)}{u}du\nonumber\\
&= & \delta \int_{t-\sigma_2(t)}^{t} \frac{du}{\sigma_2(u)}  \label{G-}\\
&\geq & \delta \frac{1}{\sigma_2(t)} \int_{t-\sigma_2(t)}^{t} du\nonumber\\
&= & \delta\nonumber
\end{eqnarray}
for $t\geq t_0$.

To estimate the size of $E_1$ we may assume that $E_1$ is unbounded.
We choose $x_1\in E_1\cap [\inf E_1,\inf E_1 +\tfrac12]$ and
put $x_1':=x_1+1/\sigma_1(T(x_1))$.
Recursively we then choose
$$x_j\in E_1\cap \left[\inf \left(E_1\cap [x_{j-1}',\infty)\right),
\inf \left(E_1\cap [x_{j-1}',\infty)\right)+2^{-j}\right]$$
and put $x_j':=x_j+1/\sigma_1(T(x_j))$.
Then
$$T(x_{j+1})\geq T(x_{j}')=T\left(x_j+\frac{1}{\sigma_1(T(x_j))}\right)
\geq T(x_j)+ \sigma_2(T(x_j))$$
and hence
$$G(T(x_{j+1}))\geq G\left(T(x_j)+ \sigma_2(T(x_j))\right)
\geq G\left(T(x_j)\right)+\eta$$
by~\eqref{G+}. Induction shows that
\begin{equation}\label{GTxj}
G(T(x_{j}))\geq G\left(T(x_1)\right)+(j-1)\eta
\end{equation}
for $j\in\N$.
In particular it
follows that $x_j\to\infty$ so that
$$E_1\subset \bigcup_{j=1}^\infty \left[x_j-2^{-j},x_j'\right].$$
Hence
$$\meas E_1\leq \sum_{j=1}^\infty \left(x_j'-x_j+2^{-j}\right)
=\sum_{j=1}^\infty\frac{1}{\sigma_1(T(x_j))}+1.$$
With $H:=\sigma_1\circ G^{-1}$ and $u_0:= G\left(T(x_1)\right)$
we deduce from~\eqref{GTxj} that
$$\sigma_1(T(x_j))=H(G(T(x_j)))\geq H(u_0+(j-1)\eta).$$
Hence
$$\sum_{j=2}^\infty\frac{1}{\sigma_1(T(x_j))}
\leq
\sum_{j=2}^\infty\frac{1}{ H(u_0+(j-1)\eta)}
\leq \frac{1}{\eta}\int_{u_0}^\infty \frac{du}{H(u)}
= \frac{1}{\eta}\int_{T(x_1)}^\infty \frac{G'(v)}{\sigma_1(v)}dv.$$
Since
$$G'(v)=\frac{1}{\sigma_2(v)}-\frac{v\sigma_2'(v)}{\sigma_2(v)^2}
\leq \frac{1}{\sigma_2(v)}$$
we obtain
$$\sum_{j=2}^\infty\frac{1}{\sigma_1(T(x_j))}\leq
\frac{1}{\eta}\int_{T(x_1)}^\infty \frac{dv}{\sigma_1(v)\sigma_2(v)}<\infty.$$
Altogether we have
$$\meas E_1\leq \frac{1}{\sigma_1(t_0)}
+\frac{1}{\eta}\int_{t_0}^\infty \frac{dv}{\sigma_1(v)\sigma_2(v)}
+1
<\infty.$$

To estimate $E_2$ we proceed similarly. We may assume that
$E_2\neq \emptyset$ and fix $R>x_0'$ so large that $E_2\cap [x_0',R]
\neq \emptyset$.
We choose
$$z_1\in
E_2\cap \left[\sup \left(E_2\cap [x_0',R]\right)-\tfrac12,
\sup \left(E_2\cap [x_0',R]\right)\right]$$
and put
$z_1':=z_1-1/\sigma_1(T(z_1))$.
Recursively we then choose
$$z_j\in E_2\cap \left[\sup \left(E_2\cap [x_0',z_{j-1}']\right)-2^{-j},
\sup \left(E_2\cap [x_0',z_{j-1}']\right)\right]$$
and
put $z_j':=z_j-1/\sigma_1(T(z_j))$,
as long as $E_2\cap [x_0',z_{j-1}'] \neq \emptyset$.
However, since
\begin{eqnarray*}
T(z_{j+1})
&\leq&
T(z_{j}')\\
&=&
T\left(z_j-\frac{1}{\sigma_1(T(z_j))}\right)\\
&\leq&
 T(z_j)-\sigma_2(T(z_j))\\
&=&
\left(1-\frac{1}{G( T(z_j))}\right) T(z_j)\\
&\leq&
\left(1-\frac{1}{G( T(z_1))}\right) T(z_j)\\
&\leq&
\left(1-\frac{1}{G( T(z_1))}\right)^j T(z_1),\\
\end{eqnarray*}
the process stops and we obtain two finite
sequences $(z_1,\dots,z_N)$ and $(z_1',\dots,z_N')$ with
$$E_2\cap [x_0',R] \subset \bigcup_{j=1}^N [z_j',z_j+2^{-j}].$$
With $y_j:=z_{N-j+1}$ we thus have
$$E_2\cap [x_0',R] \subset \bigcup_{j=1}^N [y_j',y_j+2^{j-N-1}]$$
and
$$T(y_{j})\leq T(y_{j+1})-\sigma_2(T(y_{j+1})).$$
Hence
$$G(T(y_{j}))\leq G\left( T(y_{j+1})-\sigma_2(T(y_{j+1}))\right)
\leq  G\left( T(y_{j+1})\right) -\delta$$
by~\eqref{G-} and thus
$$G(T(y_{j}))\geq G(T(y_{1}))+(j-1)\delta$$
by induction.
Now the estimate for $E_2$ is very similar to that for $E_1$.
We obtain
\begin{eqnarray*}
\meas \left(E_2\cap[x_0',R]\right)
&\leq& \sum_{j=1}^N (y_j-y_j'+2^{j-N-1})\\
&=&\sum_{j=1}^N \frac{1}{\sigma_1(T(y_j))} +\sum_{j=1}^N 2^{j-N-1}\\
&\leq& \frac{1}{\sigma_1(T(y_1))}
+\frac{1}{\delta} \int_{T(y_1)}^\infty \frac{du}{H(u)} +1\\
&\leq& \frac{1}{\sigma_1(t_0)}
+\frac{1}{\delta} \int_{t_0}^\infty \frac{du}{\sigma_1(u)\sigma_2(u)} +1
\end{eqnarray*}
and hence $\meas E_2 <\infty$.
\end{proof}
\begin{rem}
Lemma~\ref{growthlemma1a}  was proved in~\cite[Lemma 11.1]{BRS}
in the 
case  
that
$\sigma_1(t)=t^\beta$ and $\sigma_2(t)=t^{1-\alpha}$
where $0<\alpha<\beta$.
The method
of proof used here is similar, going back to a classical lemma of
Borel; see~\cite[\S 3.3]{CherryYe2001},
\cite[p.~90]{GO} and~\cite{Nevanlinna1931}.
\end{rem}
Similarly as in~\cite{BRS} we apply Lemma~\ref{growthlemma1a}
to the (right) derivative $\Phi'$ of a convex function~$\Phi$.
\begin{lemma} \label{growthlemma2a}
Let $x_0>0$ and let $\Phi:
[x_0,\infty)\to (0,\infty)$ be increasing and convex.
Let $t_0:=\Phi(x_0)$ and let $\psi:[t_0,\infty)\to(0,\infty)$ be
a differentiable function satisfying~\eqref{conv}
and~\eqref{reg} with $K>0$ and $L<2$.
Then there exists a set $E\subset [x_0,\infty)$ of finite measure
such that
\begin{equation}
\Phi(x+h)\leq\Phi(x)+\Phi'(x)h+o(1) \quad \text{for}\quad
|h|\leq \frac{1}{\sqrt{\psi(\Phi'(x))}},
\quad x\notin E,
\label{2e}
\end{equation}
uniformly as $x\to\infty$.
\end{lemma}
\begin{proof}
First we note that $\lim_{x\to\infty} \Phi'(x)$
exists since $\Phi'$ is nondecreasing.
It is easy to see that~\eqref{2e} holds without an exceptional set $E$ if
this limit is finite. Hence we assume that
$\lim_{x\to\infty} \Phi'(x)=\infty$.

Let
$$V(t):=\int_t^\infty \frac{du}{\psi(u)}$$
so that $V'(t)=-1/\psi(t)$.
We may assume that $K<1$
and apply Lemma \ref{growthlemma1a} with $T=\Phi'$ and
\begin{equation} \label{s1s2}
\sigma_1(t)=\sigma_2(t)=V(t)^{K/2} \sqrt{\psi(t)}.
\end{equation}
To show that the hypotheses of this lemma are satisfied we note that
$$
\int_{t_0}^t\frac{du}{\sigma_1(u)\sigma_2(u)}
=\int_{t_0}^t \frac{V(u)^{-K}}{\psi(u)}du
=\frac{1}{1-K}
\left(V(t_0)^{1-K}-V(t)^{1-K}\right)
$$
and thus
$$
\int_{t_0}^\infty\frac{du}{\sigma_1(u)\sigma_2(u)} <\infty.$$
We also have
\begin{equation}
\frac{t\sigma_2'(t)}{\sigma_2(t)}
=\frac{K}{2} \frac{tV'(t)}{V(t)}
+\frac12  \frac{t\psi'(t)}{\psi(t)}.
\label{2z}
\end{equation}
Since $V'(t)=-1/\psi(t)< 0$ this implies that
$$  \frac{t\sigma_2'(t)}{\sigma_2(t)}\leq
\frac12 \frac{t\psi'(t)}{\psi(t)}\leq \frac{L}{2}<1.$$
On the other hand,
since $\psi$ is increasing it follows from~\eqref{conv} that
$\psi(t)/t\to\infty$ as $t\to\infty$ and thus we find, using~\eqref{reg},
that
\begin{eqnarray*}
0 & < & -tV'(t)\\
&= & \frac{t}{\psi(t)}\\
&= & \int_t^\infty \left(\frac{u\psi'(u)}{\psi(u)^2} -\frac{1}{\psi(u)}\right)du\\
&= & \int_t^\infty \left(\frac{u\psi'(u)}{\psi(u)}\right) \frac{du}{\psi(u)}-V(t)\\
&\leq & (L-1)V(t).
\end{eqnarray*}
It follows that
$$\frac{tV'(t)}{V(t)} \geq -(L-1)$$
and this, together with~\eqref{reg} and ~\eqref{2z}, implies that
$$  \frac{t\sigma_2'(t)}{\sigma_2(t)}
\geq -\frac{K}{2} (L-1) +\frac{K}{2}=\frac{K(2-L)}{2}>0.$$
Thus the hypotheses of Lemma~\ref{growthlemma1a} are satisfied.

Next we note that~\eqref{s1s2} yields that 
$$\sigma_k(t)=o\left(\sqrt{\psi(t)}\right)$$
as $t\to\infty$ for
$k\in\{1,2\}$.
In particular, we find that
$\sigma_k(t)\leq \sqrt{\psi(t)}$ for large~$t$.
Lemma~\ref{growthlemma1a} now yields that
if $x\notin E$ is large and $0<h\leq 1/\sqrt{\psi\left(\Phi'(x)\right)}$,
then
\begin{eqnarray*}
\Phi(x+h)
&=&
\Phi(x)+\int_x^{x+h} \Phi'(u)du\\
&\leq &
\Phi(x)+\Phi'(x+h)h\\
&\leq &
\Phi(x)+\Phi'\left(x+\frac{1}{\sigma_1\left(\Phi'(x)\right)}\right)h\\
&\leq&
\Phi(x)+\left(\Phi'(x)+\sigma_2\left(\Phi'(x)\right) \right) h\\
&\leq&
\Phi(x)+\Phi'(x) h+\frac{\sigma_2\left(\Phi'(x)\right)}
{\sqrt{\psi\left(\Phi'(x)\right)}}
\end{eqnarray*}
and hence $\Phi(x)+\Phi'(x) h+o(1)$ as $x\to\infty$.
The case
$- 1/\sqrt{\psi\left(\Phi'(x)\right)}\leq h<0$ is analogous.
\end{proof}
\begin{rem}
If we apply Lemma \ref{growthlemma1a}
not to the functions defined by~\eqref{s1s2}, as we did
in the above proof, but to the functions
$\sigma_1(t)=\sigma_2(t)=\sqrt{\psi(t)}$,
then we obtain~\eqref{2e} with $o(1)$ replaced by~$1$.
Choosing $\sigma_1(t)=\sigma_2(t)=\varepsilon\sqrt{\psi(t)}$
yields~\eqref{2e} with $o(1)$ replaced by~$\varepsilon$.
\end{rem}
We apply Lemma~\ref{growthlemma2a}
to $\Phi(x)=\log M(e^x,f,D)$.
Then $\Phi'(x)=a(e^x,f,D)$.
With $r=e^x$ and $s=e^{x+h}$ we obtain
\begin{eqnarray*}
\log M(s,f,D)
&=&\Phi(x+h)\\
&\leq &
\Phi(x)+\Phi'(x)h+o(1)\\
&=&\log M(r,f,D)+a(r,f,D)\log\frac{s}{r}+o(1)
\end{eqnarray*}
for $r\notin F=\exp E$, provided that
$$\left|\log\frac{s}{r}\right|=|h|\leq \frac{1}{\sqrt{\psi(\Phi'(x))}}
=\frac{1}{\sqrt{\psi(a(r,f,D))}}.$$
This means that~\eqref{growthlemma3.1} holds for $r\notin F$
under the assumption~\eqref{growthlemma3a}.

The deduction of Theorem~\ref{th2} from the result
that~\eqref{growthlemma3.1} holds for $s$ satisfying~\eqref{growthlemma3a}
if $r\notin F$
is similar to the arguments in~\cite{BRS} where the validity
of~\eqref{growthlemma3.1} under the stronger condition~\eqref{growthlemma3.2}
is used to show that~\eqref{wvlike} holds for $z$ satisfying~\eqref{wvsize}.

\section{Proof of Theorem~\ref{th3}}
\subsection{Preliminaries}
We first note that~\eqref{div} and ~\eqref{reg}
also hold with $\psi(x)$ replaced by $\alpha\psi(\beta x)$
where $\alpha,\beta>0$, and thus it suffices to show that there
exist $\gamma,\delta>0$ such that the disk of radius
$\gamma r/\sqrt{\psi(\delta a(r,f,D))}$ around $z$ contains a zero
of $f$ if~$|z|=r$ is large.
Moreover, we see that we may assume that $\psi(t_0)\geq t_0\geq 1$.

We define
$A_1:[1,\infty)\to[t_0,\infty)$ by
\begin{equation}\label{defA1}
\log r=\int_{t_0}^{A_1(r)}\frac{du}{\psi(u)}.
\end{equation}
With $\phi:[t_0,\infty)\to[0,\infty)$,
$$\phi(t):=\int_{t_0}^{t}\frac{du}{\psi(u)}$$
we thus have $A_1(r)=\phi^{-1}(\log r)$. The function $f$
constructed will satisfy
$$a(r,f)=a(r,f,D)\sim A_1(r)$$
as $r\to\infty$. However, before we can
define the function $f$ we will have
to introduce some auxiliary functions and study their properties.

We first note
that it follows from~\eqref{reg} and the assumption that $K=1$ that
$$\log\frac{t}{t_0} \leq \log\frac{\psi(t)}{\psi(t_0)}
\leq L \log\frac{t}{t_0}.$$
Using that $\psi(t_0)\geq t_0$ we see that
\begin{equation}\label{estpsi}
t \leq \psi(t)\leq c t^L
\end{equation}
for $t\geq t_0$ and $c:=\psi(t_0)t_0^{-L}$.

It follows from~\eqref{defA1} that $A_1(r)$ is differentiable and
$A_1'(r)=\psi(A_1(r))/r$. This implies that $A_2(r):=r A_1'(r)
=\psi(A_1(r))$ is also
differentiable so that we may define $A_3(r):=r A_2'(r)$.
The functions $A_1,A_2$ and $A_3$ are thus related by
\begin{equation} \label{A23}
A_2(r)= \frac{d A_1(r)}{d\log r}=rA_1'(r)
\quad \text{and} \quad
A_3(r)= \frac{d A_2(r)}{d\log r}=rA_2'(r).
\end{equation}
Since $\psi(t)\geq t$ we have $\phi(t)\leq \log (t/t_0)$ and thus
$A_1(r)\geq t_0 r\geq r$ for $r\geq 1$.
Using~\eqref{estpsi} and recalling that~\eqref{reg}
holds with $K=1$ we find that
$A_2(r)\geq A_1(r)$ and
$$A_3(r) =r A_2'(r)=\psi'(A_1(r)) A_2(r)
\geq \frac{\psi(A_1(r))}{A_1(r)}A_2(r)\geq A_2(r).$$
Putting together the last estimates we thus have
\begin{equation} \label{A321}
A_3(r) \geq A_2(r)\geq  A_1(r)\geq r\geq 1>0
\end{equation}
for $r\geq 1$.
Combining this with~\eqref{A23} we see that $A_1$ and $A_2$ are
increasing and that $A_1(r)$ is a convex function of $\log r$.
Moreover,~\eqref{reg} yields that
\begin{equation} \label{london}
1\leq \frac{A_1(r)\psi'(A_1(r))}{\psi(A_1(r))}
=\frac{A_1(r)A_3(r)}{A_2(r)^2}\leq L.
\end{equation}
For $\rho>1$ and $r>1$ we thus have
\begin{eqnarray*}
\frac{1}{A_2(r)}-\frac{1}{A_2(\rho r)}
&= &
\int_r^{\rho r}
\frac{A_3(s)}{A_2(s)^2}\frac{ds}{s}\\
&\leq & L
\int_r^{\rho r}\frac{1}{A_1(s)}\frac{ds}{s}\\
&\leq &
\frac{L}{A_1(r)}\int_r^{\rho r}\frac{ds}{s}\\
&= &
\frac{L}{A_1(r)}\log\rho.
\end{eqnarray*}
Choosing
\begin{equation}\label{kappa}
\rho:=1+\frac{A_1(r)}{2A_2(r)}
\end{equation}
we obtain
$$
1-\frac{A_2(r)}{A_2(\rho r)}\leq
L\frac{ A_2(r)}{A_1(r)}\log\left(1+\frac{A_1(r)}{2A_2(r)}\right)
\leq
\frac{L}{2}\leq \frac{3}{5}
$$
and hence
\begin{equation}\label{kappa1}
A_2\left(r\left(1+\frac{A_1(r)}{2A_2(r)}\right)\right)
=A_2(\rho r)
\leq \frac52 A_2(r).
\end{equation}
It follows from~\eqref{estpsi} that
\begin{equation}\label{A2A1}
A_2(r)=\psi(A_1(r))\leq cA_1(r)^L
\end{equation}
so that
\begin{equation}\label{A1A2}
A_1(r)\geq  c^{-1/L} A_2(r)^{1/L}.
\end{equation}
Together with~\eqref{london} we deduce that
\begin{eqnarray*}
A_0(r)&:=&\int_1^r A_1(s) \frac{ds}{s}\\
&\geq &
\frac{1}{L}\int_1^r \left(\frac{A_1(s)}{A_2(s)}\right)^2 A_3(s) \frac{ds}{s}\\
&\geq &
\frac{1}{Lc^{2/L}} \int_1^r  A_2(s)^{2/L-2} A_3(s) \frac{ds}{s}\\
&=&
\frac{1}{c^{2/L}(2-L)}  \left( A_2(r)^{2/L-1}- A_2(1)^{2/L-1}\right).
\end{eqnarray*}
Hence
\begin{equation} \label{A0A2}
A_2(r)=o\left(A_0(r)^{L/(2-L)}\right)
\end{equation}
as $r\to\infty$.
We also note that~\eqref{london} yields
\begin{eqnarray*}
\frac{A_1(r)^2}{A_2(r)}
&=& \int_1^r A_1(s)\left(2 - \frac{A_1(s)A_3(s)}{A_2(s)^2}\right) \frac{ds}{s}
+ \frac{A_1(1)^2}{A_2(1)}\\
&\geq &
(2-L)  \int_1^r A_1(s) \frac{ds}{s}\\
&=& (2-L)  A_0(r)
\end{eqnarray*}
so that
\begin{equation} \label{london1}
\frac{A_0(r)A_2(r)}{A_1(r)^2}\leq \frac{1}{2-L}<\frac54.
\end{equation}

We now define $g:[1,\infty)\to [0,\infty)$,
$$g(r):=\int_1^r\sqrt{ A_2(s)} \frac{ds}{s}$$
so that $g'(r)=\sqrt{ A_2(r)} /r\geq 1/\sqrt{r}>0$. Thus $g$
is increasing and hence the inverse function
$h:=g^{-1}:[0,\infty)\to [1,\infty)$ exists.
We will have to use various estimates involving the derivatives
of~$h$. First we note that
$$
h'(t)=\frac{1}{g'(h(t))}=\frac{h(t)}{\sqrt{ A_2(h(t))}}
$$
and hence
\begin{equation}\label{hh'1}
\frac{h(t)}{h'(t)}=\sqrt{ A_2(h(t))}
\geq 1
\end{equation}
for $t\geq 0$ by~\eqref{A321}. We deduce that
\begin{equation}\label{hh'}
\frac{d}{dt}\left(\frac{h(t)}{h'(t)}\right)
=
\frac{ A_2'(h(t))h'(t)}{2\sqrt{ A_2(h(t))}}
=\frac{ A_3(h(t))h'(t)}{2\sqrt{ A_2(h(t))}h(t)}
=\frac{ A_3(h(t))}{2 A_2(h(t))}.
\end{equation}
Similarly we find that
$$
\frac{h''(t)}{h'(t)}=\left(1-\frac{ A_3(h(t))}{2 A_2(h(t))}\right)
\frac{1}{\sqrt{ A_2(h(t))}}
$$
which together with~\eqref{A321}, \eqref{london}
and~\eqref{A2A1} yields that
$$\left|\frac{h''(t)}{h'(t)}\right|\leq
\frac32 \frac{ A_3(h(t))}{A_2(h(t))^{3/2}}
\leq \frac{3L}{2} \frac{\sqrt{A_2(h(t))}}{ A_1(h(t))}
\leq \frac{3L\sqrt{c}}{2} A_1(h(t))^{L/2-1} =o(1)
$$
as $t\to\infty$. It follows that if $0\leq s\leq 1$, then
$$\log \frac{h'(t+s)}{h'(t)}=\int_t^{t+s} \frac{h''(u)}{h'(u)}du
=o(1)$$
and hence
\begin{equation}\label{hts}
h'(t+s)\sim h'(t)\quad\text{for}\quad 0\leq s\leq 1
\end{equation}
as $t\to\infty$.
For later use we also note that~\eqref{london},
\eqref{hh'} and~\eqref{A1A2}
yield that if $r>h(t)$, then
\begin{eqnarray}
\left|\frac{d}{dt}\left(\frac{h(t)}{h'(t)}\log\frac{r}{h(t)}\right)\right|
&=& \left|\frac{ A_3(h(t))}{2 A_2(h(t))}\log\frac{r}{h(t)}-1\right| \nonumber\\
\label{diffh}
&\leq & \frac{L}{2} \frac{ A_2(h(t))}{ A_1(h(t))}\log\frac{r}{h(t)}+1 \\
&\leq &   \frac{Lc^{1/L}}{2} A_2(h(t))^{1-1/L}\log r+1  \nonumber\\
&\leq &  \frac{Lc^{1/L}}{2}  A_2(r)^{1-1/L}\log r+1.  \nonumber\\
 \nonumber
\end{eqnarray}

Finally we shall need the following two lemmas.
\begin{lemma} \label{sumF}
Let $R>0$ and let $F:[0,R]\to\R$ be differentiable.
Then
$$\left| \sum_{k=1}^{[R]}F(k)-\int_{R-[R]}^R F(t)dt\right|
\leq R \sup_{0<t<R} |F'(t)|.$$
\end{lemma}
Here $[R]$ denotes the integer part of~$R$.
The proof is straightforward and thus omitted.
The following lemma is due to London~\cite[p.~502]{London1976}.
\begin{lemma} \label{londonla}
Let $\alpha,\beta:(0,\infty)\to (0,\infty)$
be functions such
that $\alpha$ is convex, $\beta$  is twice differentiable,
$\beta'$ is positive and unbounded and
$\beta''$ is positive and continuous.  Suppose that there exist
$L>0$ and $x_0>0$ such that
$$\frac{\beta''(x)}{\beta'(x)}
\leq L \frac{\beta'(x)}{\beta(x)}$$
for $x\geq x_0$. Suppose also that
$\alpha(x)\sim\beta(x)$ as $x\to\infty$.  Then
$\alpha'(x)\sim\beta'(x)$ as $x\to\infty$.
\end{lemma}

\subsection{The maximum modulus of $f$} \label{maxmod}
Let $h$ be as in the previous section. We define
$$ f(z):=\prod_{k=1}^\infty \left(1+\left( \frac{z}{h(k)}
\right)^{\left[\frac{h(k)}{h'(k)}\right]} \right).$$
Note that $[h'(k)/h(k)]\geq 1$ for all $k\in\N$ by~\eqref{hh'1}.

It will be apparent from the computations below that
the infinite product converges absolutely and locally
uniformly and thus defines an entire function which has
$[h(k)/h'(k)]$ equally spaced zeros on the circle of
radius $h(k)$ around~$0$.
In this section we determine the asymptotic behavior of
$\log M(r,f)$ and $a(r,f)$ as $r\to\infty$. In \S\ref{distance}
we will then show that there exist $\gamma,\delta>0$
such that if $|z|$ is sufficiently large,
then the disk of radius $\gamma|z|/\sqrt{\psi(\delta a(|z|,f))}$ contains
a zero of~$f$. Finally we will show in \S\ref{minmod} that $f$ has
only one direct tract $D$ so that $a(r,f)=a(r,f,D)$, thereby
completing the proof of Theorem~\ref{th3}.

Let now $r>0$, define $\rho$ by~\eqref{kappa} and put
$$a_k:=\log  \left(1+\left( \frac{r}{h(k)}
\right)^{\left[\frac{h(k)}{h'(k)}\right]} \right).$$
With
$$
S_1:=\sum_{k=1}^{[g(r)]} a_k,
\quad
S_2:=\sum_{k=[g(r)]+1}^{[g(\rho r)]} a_k
\quad
\text{and}
\quad
S_3:=\sum_{ k=[g(\rho r)]+1}^\infty a_k
$$
we have
$$\log M(r,f)\leq S_1+S_2+S_3.$$
First we note that
$$S_1\leq \sum_{k=1}^{[g(r)]}
\left( \left[\frac{h(k)}{h'(k)}\right]\log\frac{r}{h(k)}+\log 2 \right)
\leq \left( \sum_{k=1}^{[g(r)]}  \frac{h(k)}{h'(k)}\log\frac{r}{h(k)}  \right)
+g(r) \log 2$$
and hence Lemma~\ref{sumF} and~\eqref{diffh} yield that
$$S_1\leq \int_0^{g(r)}
\frac{h(t)}{h'(t)}\log\frac{r}{h(t)} dt
+ g(r)\left( \frac{Lc^{1/L}}{2} A_2(r)^{1-1/L}\log r+1\right) + g(r)\log 2.
$$
Substitution and integration by parts yield
\begin{eqnarray}
\int_0^{g(r)}\frac{h(t)}{h'(t)}\log\frac{r}{h(t)} dt
&=&
\int_1^r sg'(s)^2 \log \frac{r}{s} ds\nonumber \\
&=&
\int_1^r \frac{A_2(s)}{s} \log \frac{r}{s} ds\nonumber \\
&=&
\int_1^r A_1'(s) \log \frac{r}{s} ds \label{integral}\\
&=&\int_1^r A_1(s) \frac{ds}{s}-A_1(1)\log r \nonumber \\
&=& A_0(r) -t_0\log r .  \nonumber
\end{eqnarray}
Moreover,
\begin{equation}\label{gA2}
g(r)=\int_1^r \sqrt{A_2(s)}\frac{ds}{s} \leq  \sqrt{A_2(r)}\log r.
\end{equation}
Combining the above estimates we obtain
$$S_1\leq  A_0(r)+O\left( A_2(r)^{3/2-1/L}(\log r)^2\right)$$
as $r\to\infty$. Now~\eqref{A0A2} yields that
$$A_2(r)^{3/2-1/L}=A_2(r)^{(3L-2)/2L}=o\left(A_0(r)^{(3L-2)/(4-2L)}\right)$$
as $r\to\infty$. Since $L<\frac65$ we have
$$\frac{3L-2}{4-2L}<1.$$
Recalling that $A_2(r)\geq r$ we thus find that 
$$A_2(r)^{3/2-1/L}(\log r)^2=o\left(A_0(r)\right)$$
and hence that
$$S_1\leq (1+o(1)) A_0(r)$$
as $r\to\infty$. 

Next we note that $\rho\leq\frac32$ by~\eqref{A321}. Hence
\begin{eqnarray*}
S_2
&\leq & g(\rho r) \log 2\\
&\leq & \sqrt{A_2(\rho r)}\log (\rho r) \log 2\\
&\leq &  \sqrt{\frac52 A_2(r)}\left(\log r+\log \frac32\right)\log 2\\
&=&O\left(A_0(r)^{L/(4-2L)}\log r\right)\\
&=& o(A_0(r))
\end{eqnarray*}
by~\eqref{kappa1}, \eqref{A0A2} and~\eqref{gA2}.
Finally, using the abbreviation $\tau:=\log\rho$ and noting
that $h/h'$ increases by~\eqref{hh'},  we have
\begin{eqnarray*}
S_3
&\leq &
\sum_{ k=[g(\rho r)]+1}^\infty
\left( \frac{r}{h(k)}
\right)^{\left[\frac{h(k)}{h'(k)}\right]}\\
&\leq &
\sum_{ k=[g(\rho r)]+1}^\infty
\left(\frac{1}{\rho}\right)^{\frac{h(k)}{h'(k)}-1}\\
&= &
\rho\sum_{ k=[g(\rho r)]+1}^\infty
\exp\left(-\tau\frac{h(k)}{h'(k)}\right)\\
&\leq &
\rho \left(\int_{g(\rho r)}^\infty
\exp\left(-\tau\frac{h(t)}{h'(t)}\right)dt+1\right)\\
&=&
\rho  \int_{\rho r}^\infty
g'(s)\exp\left(-\tau sg'(s) \right)ds +\rho \\
&=&
\rho \int_{\rho r}^\infty
\sqrt{A_2(s)}\exp\left(-\tau\sqrt{A_2(s)}\right)\frac{ds}{s}+\rho.
\end{eqnarray*}
Using~\eqref{A321} we thus find that
\begin{eqnarray*}
S_3
&\leq&
2\rho \int_{\rho r}^\infty
\frac{A_3(s)}{2 \sqrt{A_2(s)}}
\exp\left(-\tau\sqrt{A_2(s)}\right)\frac{ds}{s}+\rho \\
&=&
\frac{2\rho}{\tau} \exp\left(-\tau\sqrt{A_2(\rho r)}\right)+\rho\\
&\leq&
\frac{2\rho}{\tau} +\rho.\\
\end{eqnarray*}
Since $1<\rho\leq\frac32$ and
\begin{equation}\label{logx}
\log x\geq (x-1)\log 2\quad \text{for}\quad 1\leq x\leq 2
\end{equation}
we have
$$\tau=\log\rho\geq (\rho-1)\log 2=\frac{A_1(r)}{2A_2(r)}\log 2
\geq \frac{\log 2}{2c^{1/L}}A_2(r)^{1/L-1}$$
by~\eqref{A1A2} and hence
$$S_3\leq
\frac{3}{\tau} +\frac32\leq
\frac{6c^{1/L}}{\log 2} A_2(r)^{1-1/L} +\frac32
=O\left(A_0(r)^{(L-1)/(2-L)}\right)
= o(A_0(r)) $$
by~\eqref{A0A2}.
Combining the estimates for $S_1$, $S_2$ and $S_3$ we conclude that
$$\log M(r,f)\leq (1+o(1) A_0(r)$$
as $r\to\infty$.

On the other hand, denoting as usual (see~\cite{GO,Hayman1964,Nevanlinna1953})
by $N(r,1/f)$ the counting function of the zeros of $f$, we have
$$\log M(r,f)\geq N\left(r,\frac{1}{f}\right)
=\sum_{|c_j|<r}\log\frac{r}{|c_j|}$$
where $c_1,c_2,\ldots$ are the zeros of~$f$.
We obtain
$$N\left(r,\frac{1}{f}\right)
=
\sum_{k=1}^{[g(r)]}
\left[ \frac{h(k)}{h'(k)} \right] \log\frac{r}{h(k)}$$
and we see as in the estimate for $S_1$ that
$$N\left(r,\frac{1}{f}\right) \geq \int_0^{g(r)}
\frac{h(t)}{h'(t)}\log\frac{r}{h(t)} dt-o(A_0(r))
=(1-o(1))A_0(r).$$
Altogether we thus have
\begin{equation} \label{MA0}
\log M(r,f)\sim A_0(r)
\end{equation}
as $r\to\infty$.
It follows from~\eqref{london1} and Lemma~\ref{londonla}, applied to $\alpha(x)=
\log M(e^x,f)$ and $\beta(x)=A_0(e^x)$, that
\begin{equation} \label{asympa}
a(r,f)\sim A_1(r)
\end{equation}
as $r\to\infty$.

\subsection{The distance to the closest zero}\label{distance}
For $z\in\C$ we denote by $\delta(z)$ the distance of $z$ to
the closest zero of $f$ and we put $d(r):=\max_{|z|=r}\delta(z)$ for $r>0$.
For $r>h(1)$ we put  $n:=[g(r)]$
so that $n\geq 1$ and $h(n)\leq r\leq h(n+1)$.
As $f$ has $[h(n)/h'(n)]$ equally spaced zeros on the circle
with radius $h(n)$ it follows that
$$d(r)\leq r-h(n)+\frac{2\pi h(n)}{\left[\frac{h(n)}{h'(n)}\right]}
\leq  h(n+1)-h(n)+7 h'(n)$$
for large~$r$.
By~\eqref{hts} we have $h'(n)\sim h'(g(r))$ and
$$h(n+1)-h(n)=\int_n^{n+1} h'(u) du \sim h'(g(r))$$
as $r\to\infty$.
Together with~\eqref{asympa} we thus find that
$$d(r)
\leq 9h'(g(r))=\frac{9}{g'(r)}=\frac{9r}{\sqrt{A_2(r)}}
=\frac{9r}{\sqrt{\psi(A_1(r))}} \leq
\frac{9r}{\sqrt{\psi\left(\frac12 a(r,f)\right)}}$$
for large~$r$.
As mentioned at the beginning of the proof, the
method thus also yields a function $f$ with $d(r)\leq r/\sqrt{\psi\left(
a(r,f)\right)}$ for large~$r$.
\subsection{The minimum modulus of $f$} \label{minmod}
For $|z|=r_n:= h\left(n+\frac12\right)$ where $n\in\N$ we have
\begin{equation}\label{lowerb}
\log|f(z)|\geq
\sum_{k=1}^{n} \log (b_k-1)
- \sum_{k=n+1}^\infty \log  \left(1+b_k \right)
\end{equation}
where
$$b_k:=\left( \frac{r_n}{h(k)}
\right)^{\left[\frac{h(k)}{h'(k)}\right]}.$$
Noting that $[g(r_n)]=n$ we see that
the estimates for $S_2$ and $S_3$ in~\S\ref{maxmod} show that
\begin{equation}\label{lowerb1}
\sum_{k=n+1}^\infty \log  \left(1+b_k \right)=o(A_0(r_n))
\end{equation}
as $n\to\infty$. To estimate the first sum on the right hand side
of~\eqref{lowerb} we note that
if $r_n \geq 2 h(k)$, then $b_k\geq 2$.
On the other hand, using~\eqref{logx} we
see that if $r_n<  2 h(k)$, then
\begin{eqnarray*}
\log b_k
&=&
\left[\frac{h(k)}{h'(k)}\right] \log  \left( \frac{r_n}{h(k} \right)\\
&=&
 \left[\frac{h(k)}{h'(k)}\right]
 \log  \left(1+\frac{h\left(n+\frac12\right) -h(k)}{h(k)} \right)\\
&\geq & \log 2  \left[ \frac{h(k)}{h'(k)} \right]
\frac{h\left(n+\frac12\right) -h(k)}{h(k)}\\
&\geq & \frac12 \frac{h\left(n+\frac12\right) -h(k)}{h'(k)}\\
&= & \frac12 \frac{1}{h'(k)} \int_{k}^{k+\frac12} h'(t)dt
\end{eqnarray*}
for large~$n$.
Using~\eqref{hts} we see  that
$\log b_k\geq \frac15$
for these values of~$k$, provided $n$ is sufficiently large.
Since $2\geq \exp\frac15$ we thus have
$b_k\geq \exp\frac15$ for all $k\leq n$ if $n$ is large.
With $B:=\frac15-\log\left( \exp\frac15 -1\right)$ we have
$$\log(b-1)\geq \log(b)-B \quad \text{for} \quad b\geq \exp\frac15$$
and thus
$$
\sum_{k=1}^{n} \log (b_k-1) \geq
\sum_{k=1}^{n} \log b_k - nB
=\sum_{k=1}^{n}
\left[\frac{h(k)}{h'(k)}\right] \log \left( \frac{r_n}{h(k)}
\right) - nB
$$
for large~$n$.
Using Lemma~\ref{sumF} and~\eqref{integral}
we conclude as in~\S\ref{maxmod} that
$$\sum_{k=1}^{n} \log (b_k-1) \geq (1-o(1))A_0(r_n).$$
Combining this with~\eqref{lowerb1} this yields
$$\min_{|z|=r_n}\log |f(z)| \geq (1-o(1))A_0(r_n).$$
In particular, $\min_{|z|=r_n}\log |f(z)|\to\infty$ as $n\to\infty$.
It follows  that $f$ has exactly one direct tract.
This completes the proof of Theorem~\ref{th3}.


\end{document}